\documentclass[11pt,a4paper]{article}
\usepackage{mathrsfs, amsthm, amssymb, bm, mathtools, thm-restate}
\usepackage{graphicx, xcolor, tikz}
\usetikzlibrary{calc}
\usepackage{caption, subcaption}
\usepackage{array}
\makeatletter
\def\th@plain{%
  \upshape 
}
\makeatother

\makeatletter
\renewenvironment{proof}[1][\proofname]{\par
  \pushQED{\qed}%
  \normalfont \topsep6\p@\@plus6\p@\relax
  \trivlist
  \item[\hskip\labelsep
        \bfseries
    #1\@addpunct{.}]\ignorespaces
}{%
  \popQED\endtrivlist\@endpefalse
}
\makeatother

\newtheorem{theorem}{Theorem}[section]
\newtheorem{lemma}[theorem]{Lemma}
\newtheorem{corollary}[theorem]{Corollary}

\newtheorem*{conjecture*}{Conjecture}
\theoremstyle{definition}

\usepackage[left=25mm,top=25mm,bottom=25mm,right=25mm]{geometry}
\usepackage[T1]{fontenc}
\usepackage[inline]{enumitem}
\usepackage[square, numbers, sort&compress]{natbib}

\usepackage[
  bookmarks=true,%
  bookmarksnumbered=true, 
  bookmarksopen=true, 
  plainpages=false,%
  colorlinks=true, 
  linkcolor=black, 
  citecolor=black,%
  anchorcolor=green,
  urlcolor=blue,
  hyperindex=true]{hyperref}

\newcommand{\etal}{et~al.\ }
\newcommand{\ie}{i.e.,\ }

\usepackage[utf8]{inputenc}
\usepackage[normalem]{ulem}
\title{Relaxed DP-3-coloring of planar graphs without some cycles}
\author{Huihui Fang \qquad Tao Wang\thanks{Corresponding author: wangtao@henu.edu.cn; https://orcid.org/0000-0001-9732-1617}\\
{\small Center for Applied Mathematics, Henan University, Kaifeng, 475004, China}}
\begin{document}
\date{}
\maketitle

\begin{abstract}
Dvo\v{r}\'{a}k and Postle introduced the concept of DP-coloring to overcome some difficulties in list coloring. Sittitrai and Nakprasit combined DP-coloring and defective list coloring to define a new coloring---relaxed DP-coloring. For relaxed DP-coloring, Sribunhung \etal proved that planar graphs without 4- and 7-cycles are DP-(0, 2, 2)-colorable. Li \etal proved that planar graphs without 4, 8-cycles or 4, 9-cycles are DP-(1, 1, 1)-colorable. Lu and Zhu proved that planar graphs without 4, 5-cycles, or 4, 6-cycles, or 4, 7-cycles are DP-(1, 1, 1)-colorable. In this paper, we show that planar graphs without 4, 6-cycles or 4, 8-cycles are DP-(0, 2, 2)-colorable.

\textbf{Keywords}: DP-coloring; Defective coloring; List coloring; Relaxed-DP-coloring

\textbf{MSC2020}: 05C15
\end{abstract}

\section{Introduction}
All graphs in this paper are simple and undirected. Assume $G$ is a plane graph, we use $V(G)$, $E(G)$, $F(G)$, and $\delta(G)$ to denote its vertex set, edge set, face set, and minimum degree in the graph $G$, respectively. We use $d(x)$ to denote the degree of $x$ for each $x \in V(G) \cup F(G)$. We say that $u$ is a $d$-vertex, $d^{+}$-vertex, or $d^{-}$-vertex if $d(u) = d$, $d(u) \geq d$, or $d(u) \leq d$, respectively. Let $b(f)$ be the boundary of a face $f$ and write $f = [v_{1}v_{2}\dots v_{d}]$, where $v_{1}, v_{2}, \dots, v_{d}$ are the boundary vertices of $f$ in a cyclic order. If $d(f)=k$ ($d(f)\geq k$ or $d(f)\leq k$), then we call $f$ a $k$-face ($k^+$-face or $k^-$-face) of $G$. A face is called a {\em simple face} if its boundary is a cycle. A cycle of length $k$ is called a {\em $k$-cycle}, and a $3$-cycle is usually called as a triangle. Two cycles or faces are {\em adjacent} if they share at least one edge, or their boundaries share at least one edge, respectively. Two adjacent cycles (or faces) $C_{1}$ and $C_{2}$ are {\em normally adjacent} if $|V(C_{1})\cap V(C_{2})| = 2$. 

We say that $L$ is a $k$-list assignment for a graph $G$ if it assigns a list $L(v)$ to each vertex $v$ of $G$ with  $|L(v)| \geq k$. If $G$ has a proper coloring $\phi$ such that $\phi(v) \in L(v)$ for each vertex $v$, then we say that $G$ is {\em $L$-colorable}. A graph $G$ is {\em $k$-choosable} if it is $L$-colorable for any $k$-list assignment $L$. The {\em list chromatic number} of $G$, denoted by $\chi_{\ell}(G)$, is the smallest integer $k$ such that $G$ is $k$-choosable. 

Dvo\v{r}\'{a}k and Postle \cite{MR3758240} introduced a generalization of list coloring. Let $G$ be a graph and $L$ be a list assignment on $V(G)$. A graph $H_{L}$, simply write $H$, is said to be a {\em cover} of $G$ if it satisfies all the following two conditions. 
\begin{enumerate}[label = (\roman*)]
\item The vertex set of $H$ is $\bigcup _{u\in V(G)}(\{u\}\times L(u))=\{(u, c): u \in V(G), c\in L(u)\}$. 
\item The edge set of $H$ is $\mathscr{M} = \bigcup_{uv \in E(G)} \mathscr{M}_{uv}$, where $\mathscr{M}_{uv}$ is a matching between the sets $\{u\} \times L(u)$ and $\{v\} \times L(v)$.
\end{enumerate}
Let $T$ be a subset of $V(H)$. If $|T \cap (\{u\}\times L(u)) | = 1$ for each vertex $u$ in $G$, then $T$ is called a {\em transversal} of $H$. When a transversal is independent, it is a {\em DP-coloring}. If every cover $H$ of $G$ with a $k$-assignment $L$ has a DP-coloring, then the least number $k$ is the {\em DP-chromatic number} of $G$, denoted by $\chi_{DP}(G)$. Note that DP-coloring is a generalization of list coloring. This implies that $\chi_{\ell}(G) \leq \chi_{DP}(G)$. Chen \etal \cite{MR3996735} proved that every planar graph without 4-cycles adjacent to 6-cycles is DP-4-colorable. Recently, it is proved that every planar graph is DP-$4$-colorable if it does not contain $i$-cycles adjacent to $j$-cycles for distinct $i$ and $j$ from $\{3, 4, 5, 6\}$, see \cite{MR3969022,MR4078909,MR3996735,MR4654340}. More sufficient conditions for a planar graph to be DP-4-colorable, see \cite{MR3881665,MR4294211,MR4089638,MR4212281}. 

In \cite{MR3962013}, Sittitrai and Nakprasit combined DP-coloring and relaxed list coloring (defective list coloring) into a new coloring as follows. Let $H$ be a cover of a graph $G$ with a $k$-assignment $L$. A transversal $T$ of $H$ is a {\em $(d_{1}, d_{2}, \dots d_{k})$-coloring} if every $(v, i) \in T$ has degree at most $d_{i}$ in $H[T]$. For any $k$-assignment $L$ and any cover $H_{L}$, if $H_{L}$ has a $(d_{1}, d_{2}, \dots, d_{k})$-coloring, then we say $G$ is {\em DP-$(d_{1}, d_{2}, \dots d_{k})$-colorable}. For defective DP-coloring, we refer the readers to \cite{MR4320363,MR4161811,MR4188936}. 

Li \etal \cite{MR4557782} proved that every planar graph without 4, 8-cycles, or 4, 9-cycles is DP-(1, 1, 1)-colorable. Lu and Zhu \cite{MR4051856} proved that every planar graph without 4, 5-cycles, or 4, 6-cycles, or 4, 7-cycles is DP-(1, 1, 1)-colorable. Sribunhung \etal \cite{MR4484592} proved that every planar graph without 4, 7-cycles is DP-(0, 2, 2)-colorable. In this paper, we prove that every planar graph without 4, 6-cycles, or 4, 8-cycles is DP-(0, 2, 2)-colorable.

To prove the conclusion, we need some new definitions. Suppose that $B$ is a condition imposed on ordered vertices. A {\em DP-B-coloring} of $H_{L}$ is a transversal $T$ with ordered vertices from left to right such that each $(v, c) \in T$ satisfies condition $B$ imposed on each element of $H$. Suppose $T$ is a transversal of a cover $H$ of $G$. We say that $T$ is a {\em DP-$B_{A}$-coloring} if the vertices in $T$ can be ordered from left to right such that:

(i) For each $(v, 1)\in T$, $(v,1)$ has no neighbor on the left.

(ii) For each $(v, c)\in T$ where $c \neq 1$, $(v, c)$ has at most one neighbor on the left and that neighbor (if it exists) is adjacent to at most one vertex on the left of $(v, c)$.

We say that $G$ is {\em DP-$B_{A}$-$k$-colorable} if every cover $H_{L}$ of a graph $G$ with a $k$-assignment $L$ has a DP-$B_{A}$-coloring. 

A graph is a {\em linear forest} if it is a forest with maximum degree at most two. It is easy to prove that a transversal $T$ is a DP-$B_{A}$-coloring only if $H[T]$ is a linear forest and $\{(v, 1): (v, 1) \in T\}$ is independent in $H$. But the inverse is not true in general. For example, let $G$ be a triangle $xyz$ and $T = \{(x, 1), (y, 2), (z, 1)\}$, where $(y, 2)$ is adjacent to $(x, 1)$ and $(z, 1)$ in a cover $H$. Observe that this $T$ has no desired ordering as in the definition DP-$B_{A}$-coloring. 

\begin{theorem}\label{DECOM}
Every planar graph without $4$- and $8$-cycles is DP-$B_{A}$-$3$-colorable.
\end{theorem}

\begin{corollary}\label{Cor1}
If $G$ is a planar graph without $4$- and $8$-cycles, then
\begin{enumerate}[label = (\roman*)]
    \item\label{c.1} \sout{$G$ is DP-$(0, 2, 2)$-colorable.}
    \item $V(G)$ can be partitioned into three sets in which each of them induces a linear forest and one of them is an independent set. 
\end{enumerate}
\end{corollary}

\begin{theorem}\label{AB}
Every planar graph without $4$- and $6$-cycles is DP-$B_{A}$-$3$-colorable.
\end{theorem}

\begin{corollary}\label{Cor2}
If $G$ is a planar graph without $4$- and $6$-cycles, then
\begin{enumerate}[label = (\roman*)]
    \item\label{c.2} \sout{$G$ is DP-$(0, 2, 2)$-colorable.}
    \item $V(G)$ can be partitioned into three sets in which each of them induces a linear forest and one of them is an independent set. 
\end{enumerate}
\end{corollary}

In order to prove results on DP-$B_{A}$-$3$-colorable graphs, Sribunhung \etal \cite{MR4484592} gave some structural results. 

\begin{lemma}[Sribunhung \etal \cite{MR4484592}]\label{delta}
If $G$ is not DP-$B_{A}$-$3$-colorable, but all its proper induced subgraphs are DP-$B_{A}$-$3$-colorable, then $\delta(G) \geq 3$.
\end{lemma}

\begin{lemma}[Sribunhung \etal \cite{MR4484592}]\label{BADNEIGHBOR}
Suppose $G$ is not DP-$B_{A}$-$3$-colorable, but all its proper induced subgraphs are DP-$B_{A}$-$3$-colorable. If a $3$-vertex $u$ in $G$ is adjacent to a $3$-vertex, then $u$ has two $5^+$-neighbors. Moreover, if $x$ is a $5$-neighbor of $u$, then $x$ has a $4^+$-neighbor.
\end{lemma}

We say that a $3$-vertex is \emph{bad} if it is adjacent to another $3$-vertex; otherwise, it is a \emph{good} $3$-vertex. 

\section{Plane graphs without 4- and 8-cycles}
Firstly, we give some structural results on plane graphs without $4$- and $8$-cycles. 
\begin{lemma}\label{ADJACENCY}
Let $G$ be a plane graph without $4$- and $8$-cycles. Then the following statements hold.
\begin{enumerate}[label = (\roman*)]
    \item\label{3F3}
    There are no adjacent $3$-faces.
    \item\label{3F5}
    If a $3$-face is adjacent to a $5$-face, then they are normally adjacent. 
    \item\label{3F6}
    If $\delta(G)\geq 3$ and a $3$-face is adjacent to a $6$-face, then they are normally adjacent.
    \item\label{3F7} 
    If $\delta(G)\geq 3$, then each $7$-face is not adjacent to any $3$-face.
    \item\label{5F5} 
    If $\delta(G)\geq3$, then there are no adjacent $5$-faces.
    \item\label{5F33} 
    If $\delta(G)\geq3$, then each $5$-face is adjacent to at most two $3$-faces.
    \item\label{6F3} 
    If $\delta(G)\geq3$, then each $6$-face is adjacent to at most one $3$-face. 
\end{enumerate}
\end{lemma}
\begin{proof}
\ref{3F3} If two $3$-faces are adjacent, then $G$ has a $4$-cycle, a contradiction.

\ref{3F5} Suppose to the contrary that a $5$-face $[v_{1}v_{2}v_{3}v_{4}v_{5}]$ is adjacent to a $3$-face $[v_1v_2u]$. Since they are not normally adjacent, $u \in \{v_{3}, v_{4}, v_{5}\}$. Then the $5$-cycle has a chord, a contradiction. 

\ref{3F6} Suppose that a $6$-face $f$ is not a simple face. Then its boundary consists of two triangles. Let $f = [u'vuwvw']$ be a $6$-face, where $[uvw]$ and $[u'vw']$ are two triangles. Observe that $G$ has no adjacent triangles. Suppose that $f$ is adjacent to a $3$-face. Then either $[uvw]$ or $[u'vw']$ bounds a $3$-face, and then there are at least two $2$-vertices, a contradiction. 

So we may assume that the $6$-face $f$ is a simple face. Suppose to the contrary that $f = [v_{1}v_{2}v_{3}v_{4}v_{5}v_6]$ is not normally adjacent to a $3$-face $[v_1v_2u]$. Then $u \in \{v_{3}, v_{4}, v_{5}, v_{6}\}$. By symmetry, we need to consider two cases: $u = v_{3}$ or $u = v_{4}$. If $u = v_{4}$, then $[v_{1}v_{2}v_{3}v_{4}]$ is a $4$-cycle, a contradiction. It follows that $u=v_3$. Since $[v_{1}v_{2}v_{3}]$ is a 3-face, we have that $v_{2}$ is a $2$-vertex, a contradiction. 
                                                                                                                                                                                                                                                                                                                                                                                                                                                                                                                                                                                                                                                                                                                                                                                                                                                                                                                                                                                                                                                                                                                                                                                                                                                                                                                                                                                                                                                                                                                                                                                                                                                                                                                                                                                                                                                                                                                                                                                                                                                                                                                                                                                                                                                                                                                                                                                                                                          
\ref{3F7} Assume that a $7$-face $f_{1}$ is adjacent to a $3$-face $f_{2}$. Observe that $f_{1}$ must be a simple face; otherwise, there is a $4$-cycle in the boundary of $f_{1}$, a contradiction. Since $\delta(G) \geq 3$ and $G$ does not have $4$-cycle, $f_{1}$ and $f_{2}$ are normally adjacent. Now, $b(f_{1}) \cup b(f_{2})$ contains an $8$-cycle, a contradiction. 

\ref{5F5} Suppose to the contrary that a $5$-face $[v_{1}v_{2}v_{3}v_{4}v_{5}]$ is adjacent to a $5$-face $[v_{1}v_{2}uvw]$. Since there is no 8-cycle, $\{u, v, w\} \cap \{v_{3}, v_{4}, v_{5}\} \neq \emptyset$. Since $\delta(G) \geq 3$ and $[v_{1}v_{2}v_{3}v_{4}v_{5}]$ has no chord, we have $\{u, w\} \cap \{v_{3}, v_{4}, v_{5}\} = \emptyset$. By symmetry, we can obtain that $\{v_{3}, v_{5}\} \cap \{u, v, w\} = \emptyset$. If $v = v_{4}$, then $[vuv_{2}v_{3}]$ is a $4$-cycle, a contradiction.   

\ref{5F33} Suppose to the contrary that a $5$-face $f$ is adjacent to three $3$-faces. If those three $3$-faces share vertices outside $f$, then $G$ has a $4$-cycle, a contradiction. Then the boundaries of these four faces form an $8$-cycle, a contradiction. Thus, each $5$-face is adjacent to at most two $3$-faces.

\ref{6F3} Suppose to the contrary that a $6$-face $f$ is adjacent to two $3$-faces. If those two $3$-faces share vertices outside $f$, then $G$ has a $4$-cycle, a contradiction. Then the boundaries of these three faces form an $8$-cycle, a contradiction. Thus, each $6$-face is adjacent to at most one $3$-face.
\end{proof}

Next, we prove the main result---\autoref{DECOM}.

Suppose to the contrary that $G$ is a minimum counterexample to the statement. By \autoref{delta}, the minimum degree of $G$ is at least three. 

A $3$-vertex is {\em special} if it is incident with a $3$-face, a $5$-face, and a $6$-face.  

\begin{lemma}\label{SPECIALVERTEX}
Let $v$ be a $3$-vertex. If $v$ is incident with a $3$-face $f_{1} = [vv_{1}v_{2}]$, a $5$-face $f_{2} = [vv_{2}v_{3}v_{4}v_{5}]$, and a $6$-face $f_{3} = [vv_{5}v_{6}v_{7}v_{8}v_{1}]$, then each of the following holds. 
\begin{enumerate}[label = (\roman*)]
\item\label{SPECIALVERTEX1} There is only one possibility for the special $3$-vertex $v$, as shown in \autoref{SPECIAL3VERTEX}, where $v_{4}$ and $v_{7}$ are identical. Furthermore, $f_{2}$ is adjacent to exactly one $3$-face, say $f_{1}$. 
\item\label{SPECIALVERTEX2} There is no other special $3$-vertex on the boundary of $f_{2}$. 
\end{enumerate}
\end{lemma}

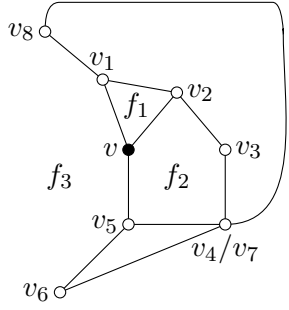
\begin{figure}
\centering
\begin{tikzpicture}[scale = 0.7]
\def\s{1}
\coordinate (v) at (135:\s);
\coordinate (v5) at (225:\s);
\coordinate (v2) at ($(v)!1!140:(v5)$);
\coordinate (v3) at ($(v2)!1!80:(v)$);
\coordinate (v4) at ($(v3)!1!140:(v2)$);
\coordinate (v1) at ($(v)!1!60:(v2)$);
\draw (v)node[left]{$v$}--(v1)node[above]{$v_{1}$}--(v2)node[right]{$v_{2}$}--(v3)node[right]{$v_{3}$}--(v4)node[below]{$v_{4}/v_{7}$}--(v5)node[left]{$v_{5}$}--cycle;
\draw (v)--(v2);
\coordinate (W) at ($(v1)!1!150:(v2)$);
\draw (v1)--(W)node[left]{$v_{8}$};
\draw (W)[out=90, in=180] to (-2*\s, 3.5*\s) [out=0, in=180] to (2*\s, 3.5*\s) [out=0, in=-90] to (2.2*\s, 3*\s) to [out= -90, in = 0](v4);
\coordinate (SW) at ($(v5)!1!-135:(v4)$);
\draw (v5)--(SW)node[left]{$v_{6}$};
\draw (SW)--(v4);
\foreach \ang in {1,...,5}
{
\node[circle, inner sep = 1.5, fill = white, draw] () at (v\ang) {};
}
\node[circle, inner sep = 1.5, fill, draw] () at (v) {};
\node[circle, inner sep = 1.5, fill = white, draw] () at (W) {};
\node[circle, inner sep = 1.5, fill = white, draw] () at (SW) {};
\node at ($(v)!0.6!-30:(v1)$) {$f_{1}$}; 
\node at (0.2*\s, 0.2*\s) {$f_{2}$};
\node at (-2*\s, 0.2*\s) {$f_{3}$};
\end{tikzpicture}
\caption{Some cases in \autoref{SPECIALVERTEX}. Note that $[v_{4}v_{5}v_{6}]$ does not bound a $3$-face.}
\label{SPECIAL3VERTEX}
\end{figure}

\begin{proof}
(i) By \autoref{ADJACENCY}\ref{3F5} and \ref{3F6}, $f_{1}$ and $f_{2}$ are normally adjacent, $f_{1}$ and $f_{3}$ are normally adjacent. Note that the $6$-face $f_{3}$ is adjacent to the $3$-face $f_{1}$, the boundary of $f_{3}$ is a $6$-cycle. Then $\{v, v_{1}, v_{2}\} \cap \{v_{3}, v_{4}, v_{5}\} = \emptyset$ and $\{v, v_{1}, v_{2}\} \cap \{v_{5}, v_{6}, v_{7}, v_{8}\} = \emptyset$. If $\{v_{3}, v_{4}\} \cap \{v_{6}, v_{7}, v_{8}\} = \emptyset$, then $[v_{1}v_{2}\dots v_{8}]$ is an $8$-cycle, a contradiction. So we may assume that $\{v_{3}, v_{4}\} \cap \{v_{6}, v_{7}, v_{8}\} \neq \emptyset$. If $v_{3} \in \{v_{6}, v_{7}, v_{8}\}$, then there is a $4$-cycle, a contradiction. It follows that $v_{3} \notin \{v_{6}, v_{7}, v_{8}\}$. If $v_{4} = v_{8}$, then $[vv_{1}v_{8}v_{5}]$ is a $4$-cycle, a contradiction. It is observed that $v_{4} \neq v_{6}$, for otherwise $v_{5}$ is a $2$-vertex. Therefore, $v_{4}$ and $v_{7}$ are identical. Note that the $3$-cycle $[v_{4}v_{5}v_{6}]$ does not bound a $3$-face, for otherwise $v_{6}$ is a $2$-vertex. Moreover, $v_{4}v_{5}$ cannot be incident with a $3$-face; otherwise, there exists a $4$-cycle with a chord $v_{4}v_{5}$. It is observed that $vv_{2}$ and $v_{4}v_{5}$ are in triangles; thus, no other edge on $f_{2}$ is contained in a triangle, for otherwise there exists an $8$-cycle. It follows that $f_{2}$ is adjacent to exactly one $3$-face. 

(ii) Since every special $3$-vertex is incident with a $3$-face, the possible other special $3$-vertex on $f_{2}$ is $v_{2}$. By \autoref{SPECIALVERTEX}\ref{SPECIALVERTEX1}, if $v_{2}$ is a special $3$-vertex, then $v_{4}$ should be identified with a vertex on the $6$-face incident with $v_{2}$, and $v_{3}v_{4}$ should be contained in a triangle, but this is impossible. 
\end{proof}

Let $\mu(x)=d(x)-4$ be the initial charge of a vertex or a face $x$, and let $\mu^{*}(x)$ denote the final charge of $x$ after the discharging process. By the Euler's formula, $|V(G)| - |E(G)| + |F(G)| =2$ and $\sum_{v\in V(G)}d(v)=\sum_{f\in F(G)}d(f)=2|E(G)|$, we can derive the following identity: $\sum_{x\in V(G)\bigcup F(G)}\mu(x)=-8$. By the following discharging rules, we shall finally get $\mu^{*}(x) \geq 0$ for all $x \in V(G) \cup F(G)$. Thus a contradiction is obtained and the proof is completed. 

The discharging rules are as follows:

\begin{enumerate}[label = \textbf{R\arabic*.}, ref = R\arabic*]
    \item\label{R1} Each $5^+$-vertex gives $\frac{1}{4}$ to each adjacent bad $3$-vertex.
    \item\label{R2} Each $5^+$-face gives $\frac{1}{3}$ to each adjacent $3$-face.
    \item\label{R3} Each $5$-face gives $\frac{1}{6}$ to each incident good $3$-vertex and $\frac{1}{12}$ to each incident bad $3$-vertex.
    \item\label{R4} Let $f$ be a $6$-face or $7$-face. Then $f$ gives $\frac{1}{2}$ to each incident good $3$-vertex and $\frac{1}{4}$ to each incident bad $3$-vertex.
    \item\label{R5} Each $8^+$-face gives $\frac{5}{6}$ to each incident good $3$-vertex and $\frac{5}{12}$ to each incident bad $3$-vertex.

Let  $\beta(f)$ be the final charge of a $5$-face $f$ after applying the rules \ref{R1}--\ref{R5}.
    \item\label{R6} If $v$ is a special $3$-vertex, then the incident $5$-face $f$ additionally sends $\beta(f)$ to $v$. 
\end{enumerate}

Now, we give a lower bound of $\beta(f)$ in \ref{R6}. 

\begin{lemma}\label{beta}
If $f$ is a $5$-face which is incident with a special $3$-vertex, then $\beta(f) \geq \frac{1}{3}$. 
\end{lemma}
\begin{proof}
By \autoref{SPECIALVERTEX}\ref{SPECIALVERTEX1}, the $5$-face is adjacent to exactly one $3$-face. If the $5$-face is incident with at most two $3$-vertices, then $\beta(f) \geq 1 - \frac{1}{3} - \frac{1}{6} \times 2 = \frac{1}{3}$ by \ref{R2} and \ref{R3}. If the $5$-face is incident with at least three $3$-vertices, then the $5$-face is incident with exactly three $3$-vertices in which two of them are bad $3$-vertices by \autoref{BADNEIGHBOR}. It follows that $\beta(f) \geq 1 -\frac{1}{3} - \frac{1}{6} - \frac{1}{12} \times 2 = \frac{1}{3}$ by \ref{R2} and \ref{R3}. 
\end{proof}

Recall that every vertex $v$ is a $3^{+}$-vertex.

Consider a good $3$-vertex $v$. If $v$ is incident with at least two $6^{+}$-faces, then $\mu^{*}(v) \geq \mu(v) + \frac{1}{2} \times 2 = 0$ by \ref{R4} and \ref{R5}. So we may assume that $v$ is incident with at least two $5^{-}$-faces. By \autoref{ADJACENCY}\ref{5F5} and \ref{3F3}, $v$ is incident with a $3$-face and a $5$-face. If $v$ is incident with an $8^+$-face, then $\mu^{*}(v) = \mu(v) + \frac{1}{6} + \frac{5}{6} = 0$ by \ref{R3} and \ref{R5}. Otherwise, $v$ is incident with a $3$-face, a $5$-face $f$, and a $6$-face by \autoref{ADJACENCY}\ref{3F7}, \ie $v$ is a special $3$-vertex. Then $\mu^{*}(v) = \mu(v) + \frac{1}{6} + \frac{1}{2} + \beta(f) \geq 0$ by \ref{R3}, \ref{R4}, \ref{R6}, and \autoref{beta}.

Consider a bad $3$-vertex $v$. By \autoref{BADNEIGHBOR}, $v$ is adjacent to two $5^{+}$-vertices. If $v$ is incident with at least two $6^+$-faces, then $\mu^{*}(v) \geq \mu(v) + \frac{1}{4} \times 2 + \frac{1}{4} \times 2 = 0$ by \ref{R1}, \ref{R4}, and \ref{R5}. Then $v$ is incident with at least two $5^{-}$-faces. By \autoref{ADJACENCY}\ref{5F5} and \ref{3F3}, $v$ is incident with a $3$-face and a $5$-face. If $v$ is incident with an $8^+$-face, then $\mu^{*}(v) = \mu(v) + \frac{1}{4} \times 2 + \frac{1}{12} + \frac{5}{12} = 0$ by \ref{R1}, \ref{R3}, and \ref{R5}. Otherwise, $v$ is incident with a $3$-face, a $5$-face $f$, and a $6$-face by \autoref{ADJACENCY}\ref{3F7}, \ie $v$ is a special $3$-vertex. Then $\mu^{*}(v) = \mu(v) + \frac{1}{4} \times 2 + \frac{1}{12} + \frac{1}{4} + \beta(f) > 0$ by \ref{R1}, \ref{R3}, \ref{R4}, \ref{R6}, and \autoref{beta}.

If $v$ is a $4$-vertex, then it is not involved in a discharging process and thus $\mu^{*}(v)=\mu(v)=0$.

Consider a $5$-vertex $v$. If $v$ is adjacent to a bad $3$-vertex, say $u$, then $v$ has a $4^+$-neighbor by \autoref{BADNEIGHBOR}. Consequently, $v$ is adjacent to at most four bad $3$-vertices. Thus $\mu^{*}(v) \geq \mu(v) - \frac{1}{4} \times 4 = 0$ by \ref{R1}.

Consider a $d$-vertex $v$ where $d \geq 6$. Then $\mu^{*}(v)\geq \mu(v) - d \times \frac{1}{4} = (d - 4) - d \times \frac{1}{4} > 0$ by \ref{R1}.

Let $f$ be a $k$-face. 

$\bullet$ $\bm{k = 3}$. It follows from \autoref{ADJACENCY}\ref{3F3} that $f$ is adjacent to three $5^{+}$-faces, and $\mu^{*}(f) = \mu(f) + 3 \times \frac{1}{3} = 0$ by \ref{R2}.

$\bullet$ $\bm{k = 4}$. Since $G$ does not contain a $4$-cycle, it does not contain a $4$-face. 

$\bullet$ $\bm{k = 5}$. It follows from \autoref{ADJACENCY}\ref{5F33} that $f$ is  adjacent to at most two $3$-faces. Suppose that $f$ is incident with a special $3$-vertex. By \autoref{SPECIALVERTEX}\ref{SPECIALVERTEX2}, $f$ is incident with exactly one special $3$-vertex. By \ref{R6} and \autoref{beta}, we get $\mu^{*}(f) = 0$. So we may assume that $f$ is not incident with a special $3$-vertex. If $f$ is incident with at most two $3$-vertices, then $\mu^{*}(f) \geq \mu(f) - \frac{1}{3} \times 2 - \frac{1}{6} \times 2 = 0$ by \ref{R2} and \ref{R3}. If $f$ is incident with at least three $3$-vertices, then $f$ is incident with exactly three $3$-vertices in which two of them are bad $3$-vertices by \autoref{BADNEIGHBOR}. It follows that $\mu^{*}(f) \geq \mu(f) - \frac{1}{3} \times 2 - \frac{1}{6} - \frac{1}{12} \times 2 = 0$ by \ref{R2} and \ref{R3}.

$\bullet$ $\bm{k = 6}$. It follows from \autoref{ADJACENCY}\ref{6F3} that $f$ is  adjacent to at most one $3$-face. If $f$ is incident with at most three $3$-vertices, then $\mu^{*}(f) \geq \mu(f) - \frac{1}{3} - \frac{1}{2} \times 3 = \frac{1}{6} > 0$ by \ref{R2} and \ref{R4}. If $f$ is incident with at least four $3$-vertices, then $f$ is incident with exactly four $3$-vertices in which all of them are bad $3$-vertices by \autoref{BADNEIGHBOR}. It follows that $\mu^{*}(f) \geq \mu(f) - \frac{1}{3} - \frac{1}{4} \times 4 = \frac{2}{3} > 0$ by \ref{R2} and \ref{R4}.

$\bullet$ $\bm{k = 7}$. If $f$ is not a simple face, then $G$ contains a $4$-cycle, a contradiction. So we may assume that $f$ is a simple face. Then $f$ is bounded by a $7$-cycle. It follows from \autoref{ADJACENCY}\ref{3F7} that $f$ is not adjacent to any $3$-face. By \autoref{BADNEIGHBOR}, $f$ is incident with at most four $3$-vertices. It follows that $\mu^{*}(f) \geq \mu(f) - \frac{1}{2} \times 4 > 0$ by \ref{R4}.

$\bullet$ $\bm{k = 8}$. If $f$ is a simple face, then $G$ contains an $8$-cycle, a contradiction. So $f$ is not a simple face, its boundary consists of a $3$-cycle and a $5$-cycle, or two $3$-cycles and a cut edge. It follows from \autoref{ADJACENCY}\ref{3F3} and \ref{5F33} that $f$ is adjacent to at most two $3$-faces. By \autoref{BADNEIGHBOR}, $f$ is incident with at most five $3$-vertices. If $f$ is incident with at most four $3$-vertices, then $\mu^{*}(f) \geq \mu(f) - \frac{1}{3} \times 2 - \frac{5}{6} \times 4 = 0$ by \ref{R2} and \ref{R5}. If $f$ is incident with five $3$-vertices, then at least four of the $3$-vertices are bad by \autoref{BADNEIGHBOR}. It follows that $\mu^{*}(f) \geq \mu(f) - \frac{1}{3} \times 2 - \frac{5}{6} - \frac{5}{12} \times 4 = \frac{5}{6} > 0$ by \ref{R2} and \ref{R5}. 

$\bullet$ $\bm{k \geq 9}$. It follows from \autoref{ADJACENCY}\ref{3F3} that a $3$-vertex is incident with at least two $4^+$-faces. If $f$ is a $9$-face incident with exactly four good $3$-vertices, then $f$ is adjacent to at most five $3$-faces and $f$ is not incident with a bad $3$-vertex, thus $\mu^{*}(f) \geq \mu(f)-\frac{1}{3}\times5-\frac{5}{6}\times4 = 0$ by \ref{R2} and \ref{R5}. So we may assume that $f$ is not a $9$-face incident with exactly four good $3$-vertices. In what follows, if $f$ is a $9$-face, then it is incident with at most three good $3$-vertices. 

Let $v_{1}, v_{2}, \dots, v_{k}$ be the vertices on the boundary of $f$, and let $f_i$ be the face sharing an edge $v_iv_{i+1}$ with $f$, where all the subscripts are taken modulo $k$. In order to easily check the final charge of $f$, we treat some transfer from an element to another element via some agents. Firstly, $f$ sends $\frac{1}{2}$ to each vertex $v_{i}$ and sends an extra $\frac{1}{6}$ to each good $3$-vertex $v_{i}$. Next, $v_{i}$ may play the role of agent. If $f_{i}$ is a $3$-face, then the agent $v_{i}$ sends $\frac{1}{6}$ to $f_{i}$, and the agent $v_{i+1}$ sends $\frac{1}{6}$ to $f_{i}$, which corresponds to \ref{R2} that $f$ sends $2 \times \frac{1}{6} = \frac{1}{3}$ to $f_{i}$. 

Suppose that $v_{i}$ is a $3$-vertex incident with $4^{+}$-vertices $v_{i-1}$ and $v_{i+1}$. Then the agent $v_{i-1}$ sends $\frac{1}{4}$ to $v_{i}$ if $f_{i-1}$ is a $4^{+}$-face; otherwise, the agent $v_{i-1}$ sends $\frac{1}{4} - \frac{1}{6}$ to $v_{i}$. Similarly, the agent $v_{i+1}$ sends $\frac{1}{4}$ or $\frac{1}{4} - \frac{1}{6}$ to $v_{i}$. Note that the $3$-vertex $v_{i}$ is incident with at most one $3$-face; thus, $f$ sends at least $(\frac{1}{2} - \frac{1}{6}) + \frac{1}{6} + \frac{1}{4} + (\frac{1}{4} - \frac{1}{6}) = \frac{5}{6}$ to $v_{i}$, which corresponds to the first part of \ref{R5}. 

Suppose that $v_{i}$ is a $3$-vertex, and one of $v_{i-1}$ and $v_{i+1}$ is also a $3$-vertex. By symmetry, let $v_{i+1}$ be a $3$-vertex. Then the agent $v_{i-1}$ sends $\frac{1}{4}$ or $\frac{1}{4} - \frac{1}{6}$ to $v_{i}$, and then $f$ sends at least $(\frac{1}{2} - \frac{1}{6}) + (\frac{1}{4} - \frac{1}{6}) = \frac{5}{12}$ in total to $v_{i}$, which corresponds to the second part in \ref{R5}. 

For each $4^{+}$-vertex $v_{i}$, when it plays the role of agent, it receives $\frac{1}{2}$ from $f$ and gives at most $2(\frac{1}{4} - \frac{1}{6}) + 2 \times \frac{1}{6} = \frac{1}{2}$. 

So we can treat $f$ sends $\frac{1}{2}$ to each vertex $v_{i}$, and $v_{i}$ maybe plays the role of agent to redistribute at most $\frac{1}{2}$ to incident $3$-faces and $3$-vertices; additionally $f$ sends an extra $\frac{1}{6}$ to each good $3$-vertices. 

If $f$ is a $9$-face incident with at most three good $3$-vertices, then $\mu^{*}(f)\geq \mu(f)-\frac{1}{2} \times 9 - \frac{1}{6} \times 3 = 0$. If $f$ is a $10^+$-face, then $f$ is incident with at most $\frac{k}{2}$ good $3$-vertices, and then $\mu^{*}(f) \geq \mu(f) - \frac{1}{2} \times k -\frac{1}{6} \times \frac{k}{2} = \frac{1}{6} > 0$.

This completes the proof. 

\section{Plane graphs without 4- and 6-cycles}
In this section, we prove the second main result---\autoref{AB}.

Assume that $G$ is a counterexample to \autoref{AB}, but all of its proper induced subgraphs are DP-$B_{A}$-$3$-colorable. By \autoref{delta}, the minimum degree of $G$ is at least three. Since $G$ has no $4$- or $6$-cycles, the following statements hold. 

\begin{lemma}\label{6G3} 
A $3$-face is not adjacent to a $6^{-}$-face. 
\end{lemma}

\begin{proof} 
If two $3$-faces are adjacent, then $G$ has a $4$-cycle, a contradiction.

Suppose that a $5$-face $[v_{1}v_{2}v_{3}v_{4}v_{5}]$ is adjacent to a $3$-face $[v_1v_2u]$. Since there is no 6-cycle, $u \in \{v_{3}, v_{4}, v_{5}\}$. But the $5$-cycle has a chord, then there is a $4$-cycle, a contradiction. 
 
Since there is no $6$-cycle in $G$, the boundary of a $6$-face consists of two triangles. Let $f = [u'vuwvw']$ be a $6$-face, where $[uvw]$ and $[u'vw']$ are two triangles. Observe that $G$ has no adjacent triangles. Suppose that $f$ is adjacent to a $3$-face. Then either $[uvw]$ or $[u'vw']$ bounds a $3$-face, and then there are at least two $2$-vertices, a contradiction. 
\end{proof}
We once again use the discharging method to complete the proof. Let $\mu(x)=d(x)-4$ be the initial charge of a vertex or a face $x$, and let $\mu^{*}(x)$ denote the final charge of $x$ after the discharging procedure. According to the Euler's formula and handshaking theorem, the sum of the initial charge is $-8$. By the following discharging rules, we should finally get $\mu^{*}(x)\geq 0$ for all $x\in V(G)\cup F(G)$. Thus a contradiction is obtained and the counterexample does not exist.

The discharging rules are as follows:
\begin{enumerate}[label = \textbf{R\arabic*.}, ref = R\arabic*]
    \item 
    \label{R.1}
    Each $5^+$-vertex gives $\frac{1}{4}$ to each adjacent bad $3$-vertex.
     \item 
     \label{R.2}
     Each $7^+$-face gives $\frac{1}{3}$ to each adjacent $3$-face.
     \item 
     \label{R.3}
     Each $5$-face gives $\frac{1}{3}$ to each incident good $3$-vertex and $\frac{1}{6}$ to each incident bad $3$-vertex.
     \item 
     \label{R.4}
     Each $6^+$-face gives $\frac{1}{2}$ to each incident good $3$-vertex and $\frac{1}{4}$ to each incident bad $3$-vertex.
\end{enumerate}

Recall that every vertex $v$ is a $3^+$-vertex.

Consider a good $3$-vertex $v$. If $v$ is not incident with a $3$-face, then $v$ is incident with three $5^+$-faces, which implies that $\mu^{*}(v) \geq \mu(v) + \frac{1}{3} \times 3 = 0$ by \ref{R.3} and \ref{R.4}. If $v$ is incident with a $3$-face, then the other two faces are $7^+$-faces by \autoref{6G3}, and then $\mu^{*}(v) = \mu(v) + \frac{1}{2} \times 2 = 0$ by \ref{R.4}.

Consider a bad $3$-vertex $v$. If $v$ is not incident with a $3$-face, then $v$ is incident with three $5^+$-faces, and then $\mu^{*}(v) \geq \mu(v) + \frac{1}{4} \times 2 + \frac{1}{6} \times 3 = 0$ by \ref{R.1}, \ref{R.3}, and \ref{R.4}. If $v$ is incident with a $3$-face, then the other two faces are $7^+$-faces by \autoref{6G3}, and then $\mu^{*}(v) = \mu(v)+\frac{1}{4} \times 2 + \frac{1}{4} \times 2 = 0$ by \ref{R.1} and \ref{R.4}.

If $v$ is a $4$-vertex, then it does not involve in a discharging process and thus $\mu^{*}(v)=\mu(v)=0$.

Consider a $5$-vertex $v$. If $v$ is adjacent to a bad $3$-vertex, say $u$, then $v$ has a $4^+$-neighbor by \autoref{BADNEIGHBOR}. Consequently, $v$ is adjacent to at most four bad $3$-vertices. Then $\mu^{*}(v) \geq \mu(v) - 4 \times \frac{1}{4} = 0$ by \ref{R.1}.

Consider a $d$-vertex $v$ where $d \geq 6$. Then $\mu^{*}(v)\geq \mu(v) - d \times \frac{1}{4} = (d - 4) - d \times \frac{1}{4} > 0$ by \ref{R.1}.

Let $f$ be a $k$-face.

$\bullet$ $\bm{k = 3}$. It follows from \autoref{6G3} that $f$ is adjacent to three $7^{+}$-faces. Then $\mu^{*}(f)=\mu(f)+3\times\frac{1}{3}=0$ by \ref{R.2}.

$\bullet$ $\bm{k = 4}$. Since $G$ does not contain a $4$-cycle, it does not contain a $4$-face.

$\bullet$ $\bm{k = 5}$. It follows from \autoref{6G3} that $f$ is not adjacent to any $3$-face. If $f$ is incident with at most two $3$-vertices, then $\mu^{*}(f) \geq \mu(f) - \frac{1}{3} \times 2 > 0$ by \ref{R.3}. If $f$ is incident with at least three $3$-vertices, then it is incident with exactly three $3$-vertices in which two of them are bad $3$-vertices by \autoref{BADNEIGHBOR}. It follows that $\mu^{*}(f) \geq \mu(f) - \frac{1}{6} \times 2 - \frac{1}{3} > 0$ by \ref{R.3}.

$\bullet$ $\bm{k = 6}$. It follows from \autoref{6G3} that $f$ is not adjacent to a $3$-face. Since there are no $6$-cycles in $G$, the boundary of $f$ consists of two triangles, and $f$ is incident with at most four $3$-vertices. If $f$ is incident with at most three $3$-vertices, then $\mu^{*}(f) \geq \mu(f) - \frac{1}{2} \times 3 = \frac{1}{2} > 0$ by \ref{R.4}. If $f$ is incident with four $3$-vertices, then each $3$-vertex is bad, and then $\mu^{*}(f)=\mu(f)-\frac{1}{4}\times4=1>0$ by \ref{R.4}.

$\bullet$ $\bm{k \geq 7}$. Similar to the proof in \autoref{DECOM}, we can treat $f$ sends $\frac{3}{7}$ to each incident vertex and redistribute at most $\frac{3}{7}$ to incident $3$-faces and $3$-vertices. Thus, $\mu^{*}(f) \geq \mu(f) - \frac{3}{7} \times k \geq 0$.

This completes the proof. 

\bigskip
\noindent\fcolorbox{blue}{red!30}{\parbox{\textwidth}{\textbf{Added Note to the Published Version}: Upon further examination of the definition of DP-$B_{A}$-coloring, it is evident that for every $(v, 1) \in T$, there are no neighbors in $T_{i} \coloneqq T \cap \{(u, 1): (u, 1) \in T\}$. However, it is possible for $(v, 1)$ to have other neighbors in $T$. As a result, it is incorrect to claim that DP-$B_{A}$-$3$-colorability implies DP-$(0, 2, 2)$-colorability. Consequently, the assertions made in \autoref{Cor1}\ref{c.1} and \autoref{Cor2}\ref{c.2} are invalid. 

\mbox{---} On the other hand, we claim that $H[T]$ is a forest if $T$ is a DP-$B_{A}$-coloring. Let $x_{1}, x_{2}, \dots, x_m$ represent the desired ordering of the vertices in $T$. Suppose that $x_{r}$ is in a cycle with the largest subscript. By the choice of the subscript $r$, the vertex $x_{r}$ must have at least two neighbors on the left, but this contradicts the definition of DP-$B_{A}$-coloring. A DP-$(\mathcal{F}_{d_{1}}, \mathcal{F}_{d_{2}}, \dots, \mathcal{F}_{d_{k}})$-coloring is defined as a transversal $T$, where $H[T]$ is a forest and every vertex $(v, i) \in T$ has at most $d_{i}$ neighbors in $T$. We say that a graph $G$ is DP-$(\mathcal{F}_{d_{1}}, \mathcal{F}_{d_{2}}, \dots, \mathcal{F}_{d_{k}})$-colorable if every cover of $G$ with a $k$-list assignment has a DP-$(\mathcal{F}_{d_{1}}, \mathcal{F}_{d_{2}}, \dots, \mathcal{F}_{d_{k}})$-coloring. Similarly, we can define a weak version of DP-$(\mathcal{F}_{d_{1}}, \mathcal{F}_{d_{2}}, \dots, \mathcal{F}_{d_{k}})$-coloring.  A weak DP-$(\mathcal{F}_{d_{1}}, \mathcal{F}_{d_{2}}, \dots, \mathcal{F}_{d_{k}})$-coloring is a transversal $T$, where $H[T]$ is a forest and every vertex $(v, i) \in T$ has at most $d_{i}$ neighbors in $T_{i}$. We say that a graph $G$ is weakly DP-$(\mathcal{F}_{d_{1}}, \mathcal{F}_{d_{2}}, \dots, \mathcal{F}_{d_{k}})$-colorable if every cover of $G$ with a $k$-list assignment has a weak DP-$(\mathcal{F}_{d_{1}}, \mathcal{F}_{d_{2}}, \dots, \mathcal{F}_{d_{k}})$-coloring. It can be easily verified that DP-$B_{A}$-$3$-colorability implies weak DP-$(\mathcal{F}_{0}, \mathcal{F}_{2}, \mathcal{F}_{2})$-colorability. For further discussions, we refer the reader to [P.~Sittitrai, K.~M. Nakprasit and K.~Nakprasit, A weak {DP}-partitioning of planar graphs without 4-cycles and 6-cycles, Bull. Malays. Math. Sci. Soc. 46~(4) (2023) 141]. }}

\section*{Declarations}

\noindent\textbf{Conflict of Interest} The authors have no relevant financial or non-financial interests to disclose.

\end{document}